\documentclass{amsart}
\usepackage{amssymb, latexsym}

\newtheorem{theorem}{Theorem}
\newtheorem{lemma}{Lemma}
\newtheorem{remark}{Remark}
\newtheorem{definition}{Definition}
\newtheorem{example}{Example}

\begin{document}

\title{On Intersections of Conjugate Subgroups}
\author{Rita Gitik}
\email{ritagtk@umich.edu}
\address{ Department of Mathematics \\ University of Michigan \\ Ann Arbor, MI, 48109}  

\date{\today}

\begin{abstract} We define a new invariant of  a conjugacy class of subgroups which we call the weak width and prove that a quasiconvex subgroup of a negatively curved group has finite weak width in the ambient group.
Utilizing the coset graph and the geodesic core of a subgroup we give an explicit algorithm for constructing a finite generating set for an intersection of a quasiconvex subgroup of a negatively curved group with a conjugate. Using that algorithm we construct  algorithms for computing the weak width, the width and the height of a quasiconvex subgroup of a negatively curved group. These algorithms decide if a quasiconvex subgroup of a negatively curved group is  almost malnormal in the ambient group.
We also explicitely compute a quasiconvexity constant of the intersection of two quasiconvex subgroups and give examples demonstrating that height, width and weak width are different invariants of a subgroup. 
\end{abstract}
\maketitle

\section{Introduction}

A subgroup $H$ of $G$ is malnormal in $G$ if for any $g \in G$ such that $g \notin H$ the intersection $H \cap g^{-1} H g$ is trivial. $H$ is almost malnormal in $G$ if for any $g \in G$ such that $g \notin H$ the intersection $H \cap g^{-1} H g$ is finite.
Most subgroups are neither normal nor malnormal, so the study of the intersection pattern of conjugates of a subgroup is an interesting old problem. It is closely connected to the study of the behavior of different lifts of subspaces of topological spaces in  covering spaces. 
Malnormality of a subgroup has been generalized in different ways. One of them, namely the height, introduced  in \cite{G-R}, has been used by Agol in \cite{Ag}  and \cite{A-G-M} in his proof of  Thurston's conjecture that $3$-manifolds are virtual bundles.

At the end of this section we introduce yet another generalization of malnormality. It is a new invariant of the conjugacy class of a subgroup $H$ of $G$, which we call the weak width of a subgroup. Like malnormality, the weak width measures only the cardinality of the intersections of $H$ with its conjugates in $G$.

In section $2$ we review the definitions and the basic properties of the width and the height of a subgroup. 

In section $3$ we give an algorithm for constructing a generating set consisting of short elements for the intersection of a quasiconvex subgroup of a negatively curved group with a conjugate. Using that algorithm we show how to determine if the aforementioned intersection is infinite. The algorithm utilizes the geodesic core of a subgroup introduced by the author in \cite{Gi1} and \cite{Gi2}. 

In section $4$ we show that quasiconvex subgroups of negatively curved groups have finite weak width and introduce an algorithm for computing   the weak width of a quasiconvex subgroup of a negatively curved group in finite time, which might simplify Agol's proof. This algorithm also decides  if a quasiconvex subgroup $H$ of a negatively curved group $G$ is almost malnormal.  

Bridson and Wise showed in \cite{B-W} that the malnormality of a subgroup in a negatively curved group is undecidable. At the end of section $4$ we give an indication that malnormality might be undecidable even for a quasiconvex subgtoup of a negatively curved group.

In section $5$ we introduce algorithms for computing the height and the width of a quasiconvex subgroup of a negatively curved group in finite time. We also explicitely compute a quasiconvexity constant of the intersection of two quasiconvex subgroups.

Kharlampovich, Miasnikov and Weil constructed
a different algorithm for computing the height of a quasiconvex subgroup of a negatively curved group and deciding its almost malnormaity in a recent preprint \cite{K-M-W}. Their algorithm also utilizes the geodesic core of a subgroup which they reintroduce under the name "Stallings graph".

 In section $6$ we give examples showing that height, width, and weak width are different invariants of a subgroup.

\begin{remark} \label{R:ShortRep}
 Note that if $g_i \in Hg_jH$, hence $g_i=h_1g_jh_2$ with $h_1$ and $h_2$ in 
$H$, then  $H \cap g_i^{-1}Hg_i = H \cap (h_1g_jh_2)^{-1}H(h_1g_jh_2)= H \cap (h_2^{-1}g_j^{-1}Hg_jh_2)=
h_2^{-1}(H \cap g_j^{-1}Hg_j)h_2$. So the cardinality of the set $H \cap g_i^{-1}Hg_i$ is equal to the cardinality of the set $H \cap g_j^{-1}Hg_j$. 
\end{remark}

Let $H$ be a subgroup of a group $G$. We say that the elements $\{g_i| 1 \leq i \leq n \}$ of $G$ belong to different double cosets of $H$ if $Hg_iH \neq Hg_jH$ for $i \neq j$.

Remark \ref{R:ShortRep} motivates the following definition.

\begin{definition}
We say that the weak width of an infinite subgroup $H$ of $G$ in $G$, denoted $WeakWidth(H,G)$, is $n$ if there exists a collection of $n$ elements  $\{ g_1 =1_G, g_2, \cdots g_n \}$ of $G$ belonging to different double cosets of $H$ such that the intersection $H \cap g_i^{-1} H g_i$ is infinite for all $1 \leq i \leq n$  and $n$ is maximal possible. We define the weak width of a finite subgroup of $G$ to be $0$.
\end{definition}

Note that if $H$ is infinite then $WeakWidth(H,G)=1$ if and only if $H$ is almost malnormal in $G$. 

If $WeakWidth(H,G)=n$, then in any set of $n+1$  conjugates of $H$ by elements in different double cosets of $H$,  $ \{H, g_2^{-1} H g_2,\cdots, g_{n+1}^{-1} H g_{n+1} \}$ there exists an element
$g_i^{-1} H g_i$ which has finite intersection with $H$.

\bigskip

\textbf{Acknowledgment}

The author would like to thank Shmuel Weinberger for his support.

\section{Height and Width.}

\begin{definition}
Let $H$ be a subgroup of a group $G$. We say that the elements $\{g_i| 1 \leq i \leq n \}$ of $G$ belong to different cosets of $H$ if $Hg_i \neq Hg_j$ for $i \neq j$. 
\end{definition}

If $g_i$ and $g_j$ belong to the same coset of $H$, then $g_jg_i^{-1} \in H$, so $g_jg_i^{-1} H g_ig_j^{-1}
=H$, hence $g_i^{-1} H g_i = g_j^{-1} H g_j$. Therefore it is interesting to investigate the intersections of a family of conjugates of $H$ only if the conjugating elements belong to different cosets of $H$. However, the following example shows that the conjugates of a subgroup $H$ by elements in different cosets of $H$ and different double cosets of $H$ need not be distinct. 

\begin{example}
 Let $G=<a_1, a_2|a_1a_2=a_2a_1>$ be a free abelian group of rank $2$ and let $H=<a_1>$ be a subgroup of $G$. As $Ha_2H =H a_2 \neq H$, the elements $1_G$ and $a_2$ belong to different cosets and to different double cosets of $H$, but   $a_2^{-1} H a_2 =H$. 
\end{example}

The following definitions were introduced in \cite{G-R} and \cite{G-M-R-S}.

\begin{definition}
We say that the height of an infinite subgroup $H$ of $G$ in $G$, denoted by $Height(H,G)$, is $n$  if there exists a collection of $n$ conjugates of $H$ in $G$ by elements in different cosets of $H$ such that the intersection of all the elements of the collection is infinite and $n$ is maximal possible.  We define the height of a finite subgroup of $G$ to be $0$.
\end{definition}

Note that if $Height(H,G) =n$ then the intersection of any set of $n+1$  conjugates of $H$ by elements in different cosets of $H$ is finite.
It was shown in \cite{G-M-R-S} that quasiconvex subgroups of negatively curved groups have finite height in the ambient group.

\begin{definition}
We say that the width of an infinite subgroup $H$ of $G$ in $G$, denoted by $Width(H,G)$, is $n$ if there exists a collection of $n$ conjugates of $H$ by elements in different cosets of $H$ such that the intersection of any two elements of the collection is infinite and $n$ is maximal possible.  We define the width of a finite subgroup of $G$ to be $0$.
\end{definition}

Note that if $Width(H,G)=n$ then in any set of $n+1$  conjugates of $H$ by elements in different cosets of  $H$ there exist two elements with finite intersection. It was shown in \cite{G-M-R-S} and, later, by Hruska and Wise in \cite{H-W} that quasiconvex subgroups of negatively curved groups have finite width in the ambient group.

It follows from the above definitions that $Width(H,G)$, $WeakWidth(H,G)$ and $Height(H,G)$ are invariants of the conjugacy class of $H$ in $G$ and that $Height(H, G) \leq Width(H, G)$. However, there is no obvious relationship between  
$WeakWidth(H, G)$ and $Width(H, G)$. We will address this question in section $5$.

Infinite normal subgroups of infinite index have infinite height, width, and weak width in the ambient group. More generally, if an infinite subgroup has infinite index in its  normalizer, then the
 subgroup has infinite height, width, and weak width in the ambient group.
 
 If  $H$ is infinite, then $H$ is almost malnormal in $G$ if and only if $Height(H,G)= Width(H,G)=WeakWidth(H,G)=1$. Also, almost malnormal subgroups of a torsion-free group are malnormal.

\section{An algorithm deciding if the intersection of a quasiconvex subgroup of a negatively curved group with a conjugate is finite}

\begin{remark}\label{R:Algorithm For Pairs}Let $H$ be a subgroup of $G$, and let $h \in H$ and $g \notin H$. We want to decide if the intersection $H \cap g^{-1} H g$ is finite, however, in general,  finiteness of a group is undecidable. Hence we restrict ourselves to the special case of $H$ being quasiconvex and $G$ being negatively curved. In this case the conjugate $g^{-1} H g$ is also quasiconvex, hence the intersection $H \cap g^{-1} H g$ is quasiconvex in $G$, therefore the group $H \cap g^{-1} H g$ is negatively curved.

Noel Brady in \cite{Br} and Bogopolskii and Gerasimov in \cite{B-G} showed that the orders of finite subgroups of a $\delta$-negatively curved group $G$ generated by a finite set $X$ is bounded by a constant $C = C(\delta, |X|)$.  Dahmani and Guirardel showed in \cite{D-G} that the isomorphism problem for negatively curved groups is decidable, so  we can determine whether  $H \cap g^{-1} H g$ is finite by checking if it is isomorphic to any finite group of order less than $C$. 
\end{remark}

\begin{remark}\label{R:Delta}
Given a finite presentation $<X|R>$ for a group $G$, it is undecidable if $G$ is negatively curved. However, if $G$ is known to be negatively curved, the negative curvature constant $\delta$ can be determined, (cf. \cite{E-H} and \cite{Pa}), so in this paper we can assume that $\delta$ is given.
\end{remark}

We  formalize Remark \ref{R:Algorithm For Pairs} as follows.

\textbf{An algorithm deciding if the intersection of a quasiconvex subgroup of a negatively curved group with a conjugate is finite}

Input: a finite presentation $<X|R>$ for a group $G$ and a constant $\delta$ (not necessarily minimal) of negative curvature of $G$, a finite generating set for a subgroup $H \subset  G$ and a quasiconvexity constant $K$ (not necessarily minimal) of $H$, and an element $g \in G$ which is not in $H$.

Output: a finite group isomorphic to $H \cap g^{-1} H g$ or a statement that the intersection is infinite.

\begin{enumerate}
\item Using Theorem \ref{T:GenSerFotINtersections} (on next page) exhibit explicitly a finite generating set for $H \cap g^{-1} H g$.
\item Find a bound $C$ for the order of finite subgroups of $G$.
\item Make a list $L'$ of all finite groups with fewer than $C$ elements.
\item Check if $H \cap g^{-1} H g$ is isomorphic to an element of $L'$. If positive, output the intersection. If negative, output the statement that the intersection is infinite.
\end{enumerate}

The rest of the section provides the proof of Theorem \ref{T:GenSerFotINtersections}.

We will use the following notation.

Let $X$  be a set and let $X^* = \{x,x^{-1} |x \in X \}$, where for $x \in X$ we define $(x^{-1})^{-1} =x$.
Let $G$ be a group generated by the set $X^*$.
As usual, we identify the word in $X^*$ with the corresponding element in $G$. We denote the equality of two words in $X^*$ by $" \equiv "$.

 Let $Cayley(G)$ be the Cayley graph of $G$ with respect to the generating set $X^*$. The set of vertices of $Cayley(G)$ is $G$,  the set of edges of $Cayley(G)$ is $G \times X^*$, and the edge $(g,x)$ joins the vertex $g$ to $gx$. Note that $Cayley(G)$ can be effectively constructed if and only if the word problem in $G$ is solvable.
 
\begin{definition} 
The label of the path  
$p=(g,x_1)(gx_1,x_2)  \cdots (gx_1x_2 \cdots x_{n-1},x_n)$ 
in $Cayley(G)$ is the word $Lab (p) \equiv x_1 \cdots x_n $.
The length of the path $p$, denoted by $|p|$, is the number of edges forming it. The inverse of a path $p$ is denoted by $\bar p$, 
\end{definition}

Let $H$ be a subgroup of $G$. The coset graph of $H$ in $G$, denoted by $Cayley(G,H)$, provides a good geometric insight into the structure of intersections of conjugates of $H$ in $G$. Recall that the set of vertices of $Cayley(G,H)$ is the set of the  cosets  of $H$ in $G$, namely $\{ Hg \}$. 
The set of edges of $Cayley(G,H)$ is
$ \{ Hg \}  \times X^*$. An edge $ (H g, x)$ begins at the vertex  $Hg$ and ends at the vertex 
$Hgx $. 
Note that $H$ acts on the Cayley graph of $G$ by left 
multiplication, and $Cayley(G,H)$ can be defined as the quotient of $Cayley(G)$ by this action. $Cayley(G,H)$ can be effectively constructed if and only if the generalized word problem for $H$ in $G$ is solvable.

\begin{definition} 
The label of a path $p = (Hg_1,x_1) (Hg_1x_1, x_2) \cdots (Hg_1x_1 \cdots x_{n-1},x_n)$ in $Cayley(G,H)$ is the word $Lab (p) \equiv x_1 x_2 \dots x_n $. 
The length of the path $p$, denoted by $|p|$, is the number of edges forming it. The inverse of a path $p$ is denoted by $\bar p$. 
\end{definition}

As usual, we identify the word $Lab(p)$ with the corresponding element in $G$.

\begin{definition} 
Let $\pi_{H} :Cayley(G) \rightarrow Cayley(G,H)$ be
the projection map: $\pi_{H} (g)=Hg$ and
$\pi_{H} (g,x)=(Hg,x)$.
A geodesic in $Cayley(G,H)$ is the image of a 
geodesic in 
$Cayley(G)$ under the projection $\pi_{H}$.
\end{definition} 

Hence, if $H$ is non-trivial, $Cayley(G,H)$ contains closed geodesics (i.e. geodesics with initial and terminal vertices coinciding) of positive length.

By definition of the coset graph,  a path $p$ in $Cayley(G, H)$ which begins at
$H \cdot 1$  ends at $ H \cdot Lab (p)$,
so a path $p$ beginning at $H \cdot 1$ is a loop (i.e. the initial and the terminal vertices of $p$ coincide), if and only if  $Lab(p) \in H $.

Moreover, a path $p$ in $Cayley(G,H)$ which begins at the vertex $Hg$ ends at the vertex $H \cdot gLab(p)$, hence such $p$ is a loop   if and only if $HgLab(p)=Hg$, which happens if and only if $ g Lab(p) g^{-1} \in H$.

Let $h \in H$ and $g \in G$. In order to check if $g^{-1}h g$ belongs to  $H $, consider a path $p$ in $Cayley(G,H)$ which begins at $H \cdot 1$ and has the following decomposition: $p=q \gamma t$ with $Lab(q)= g^{-1}, Lab(\gamma) = h$ and $Lab(t) \equiv Lab(q) ^{-1}= g$. Note that $g^{-1}h g \in H $ if and only if
$p$ is a loop.  As $Lab(t) \equiv Lab(q) ^{-1}$, $p$ is a loop if and only if $t$ coincides with $\bar q$, so $t = \bar q$. This happens if and only if  $\gamma$ is a loop.  As $q$ begins at $H \cdot 1$, it ends at $H \cdot Lab(q)=H g^{-1}$, so $\gamma$ begins at $H g^{-1}$. Hence $g^{-1}h g$ belongs to  $H$ if and only if $Cayley(G,H)$ contains  a loop beginning at $Hg^{-1}$ labeled with $h$. 

We assume for the rest of the paper that the subgroup $H$ is given by specifying a generating set $H_0$, which will be assumed to be finite if $H$ is quasiconvex.  Note that the set $g^{-1}H_0g$ is a finite generating set for the conjugate  $g^{-1}Hg$. 
 
Denote the $K$-neighborhood of a set $S$ by $N_K(S)$.
 
 \begin{theorem}\label{T:GenSerFotINtersections}
 Let $G$ be a $\delta$-negatively curved group generated by a finite set $X^*$. Let $H$ be a $K$-quasiconvex subgroup of $G$, and let $g$ be an element of $G$ which does not belong to $H$.
Let $M$ be the number of vertices in  $N_{2\delta + K + |g|}(H \cdot 1) \subset Cayley(G,H)$. 
There exists a finite, effectively described, generating set for
the subgroup $H \cap g^{-1}Hg$ consisting of elements shorter than $2 |g| + 2M^2 +1$.
\end{theorem}

\begin{proof}
Note that as the set $X$ has finitely many, say $l$, elements, it follows that $M \leq 2l \cdot (2l-1)^{2\delta + K + |g| -1}$ where equality might  hold when $G$ is free.

Let $h$ be a non-trivial element of $H$ such that $ g^{-1}hg \in H$, hence $g^{-1}hg \in H  \cap g^{-1}Hg$.
Applying Lemma \ref{L:ConstrGenForIntes} (which follows on next page) finitely many times, we can find a decomposition $h=h_1 \cdots h_m$ such that $h_i \in H, g^{-1}h_ig \in H$ and $|h_i| <  2M^2 +1$ for all $1 \le i \le m$. Then $g^{-1}hg =(g^{-1}h_1g) \cdots (g^{-1}h_mg)$. Note that $|g^{-1}h_ig| < 2 |g| + 2M^2 +1$, hence the subgroup $H \cap g^{-1}Hg$ is generated by the set $S$ of elements shorter than $2|g| + 2M^2 +1$. In order to describe $S$ proceed as follows: start by making a list $S_0$ of all elements in $G$ shorter than $2|g| + 2M^2 +1$. This can be accomplished in finite time because negatively curved groups have solvable word problem. A practical way to do so is to implement the Dehn algorithm (cf., for example \cite{B-H} p.442) on the set of all reduced words in $X^*$ of length not greater than $2|g| + 2M^2 +1$. As $H$ and  $ g^{-1} H g$ are quasiconvex in $G$, they have a solvable membership problem, (which we will address at the end of this section).                                       Hence we can determine the intersection $S_0 \cap( H \cap g^{-1} H g)$ by first determining the intersection  $S_1= (S_0 \cap H)$ and then determining the intersection $S_1 \cap g^{-1} H g$ which is the desired generating set $S$ of $ H \cap g^{-1} Hg$.

\end{proof}

We will need the following characterization of quasiconvexity, observed in \cite{Gi2}.

\begin{remark}\label{R:GeodesicsInCosetGraph}
 A subgroup $H$  of a group $G$ is  $K$-quasiconvex in $G$ if and only if any geodesic in  $Cayley(G,H)$ which begins at the basepoint $H \cdot 1$ and is labeled by an element of $H$ belongs to  $N_K(H \cdot 1)$. 
\end{remark}

\begin{lemma}\label{L:ConstrGenForIntes}
Let $H, g$ and $M$ be as in the satement of Theorem \ref{T:GenSerFotINtersections}, and
let $h$ be a non-trivial element of $H$ such that $ g^{-1}hg \in H$. If $|h| > M^2$ then $h=h_1h_2$ with $h_1 \in H, h_2 \in H, g^{-1}h_1g \in H, g^{-1}h_2g \in H,  |h_1| <  2M^2 +1$, and $|h_2| < |h|$.   
\end{lemma}

\begin{proof}
There exists $h_0 \in H$ such that $h_0=g^{-1}hg$. Consider a closed geodesic $4$-gon $pqrs$ in $Cayley(G)$ with $p$ beginning at the basepoint $1_G$ such that $Lab(p)=h_0, Lab(q)=g^{-1}, Lab(r)=h$ and $Lab(s) \equiv Lab(q) ^{-1} =g$. As $r \subset N_{2\delta}(spq)$ and $|s|=|q|=|g|$, it follows that $r \subset N_{2\delta + |g|}(p)$. 

Consider the projection $\pi_H(pqrs)$ in $Cayley(G,H)$. As $Lab(p) \in H$ and $p$ begins at $1_G$, $\pi_H(p)$ is a  loop beginning at $H \cdot 1$. Remark \ref{R:GeodesicsInCosetGraph} states that $\pi_H(p) \subset N_K(H \cdot 1)$. Then, as the projection map does not increase distances, $\pi_H(r) \subset N_{2\delta + |g|}(\pi_H(p)) \subset N_{2\delta + |g|+ K}(H \cdot 1)$. As $Lab(q) \equiv Lab(s)^{-1}$, $\pi_H(q)$ begins at $H \cdot 1$ and $\pi_H(s)$ ends at $H \cdot 1$, it follows that $\pi_H(s) = \overline{\pi_H(q)}$ and $\pi_H(r)$ is a loop beginning at $H g^{-1}$.

Let $t$ be a closed geodesic in $Cayley(G,H)$ beginning at $H \cdot 1$ with $Lab(t) \equiv Lab(r)=h$. Let $|t|=|\pi_H(r)|=n$. Let $v_1=H \cdot 1, v_2, \cdots, v_n=H \cdot 1$ be the vertices of $t$ in order along $t$, and let $w_1=Hg^{-1}, w_2, \cdots, w_n=Hg^{-1}$ be the vertices of $\pi_H(r)$ in order along $\pi_H(r)$. Consider the set of pairs of vertices $ \{ (v_i,w_i)| 1 \le i \le n \}$. If $n > M^2$ then there exist a pair of  indexes $1 \le i < j \le M^2 +1$ such that $(v_i, w_i) = (v_j, w_j)$. This means that the loops $t$ and  $\pi_H(r)$ have self-intersections "in the same place". To be precise it means that there exists a decomposition $t=t_1 t_2 t_3$ with $t_1$ ending at $v_i$, $t_2$ beginning at $v_i$ and ending at $v_j$ and $t_3$ beginning at $v_j$ with $t_2$ a loop. There also exists a decomposition $\pi_H(r)= \pi_H(r)_1 \pi_H(r)_2 \pi_H(r) _3$ with $\pi_H(r)_1$ ending at $w_i$,  $\pi_H(r)_2$ beginning at $w_i$ and ending at $w_j$ and $\pi_H(r)_3$ beginning at $w_j$ with $\pi_H(r)_2$ being a loop. By construction $Lab(t_2) \equiv Lab(\pi_H(r)_2)$. 

Let $h_1=Lab(t_1t_2 \bar t_1)$ and $h_2=Lab(t_1t_3)$, then $h=Lab(t)=Lab(t_1t_2\bar t_1)Lab(t_1t_3)=h_1h_2$.
We will show that $h_1$ and $h_2$ fulfill the requirements of the lemma.

 As $t_2$ is a loop, it follows that the paths $t_1 t_2 \bar t_1$ and  $t_1t_3$ are loops beginning at $H \cdot 1$, hence $h_1 \in H$ and $h_2 \in H$. 
 
 As  $\pi_H(r)_2$ is a loop, the paths $\pi_H(r)_1 \pi_H(r)_2 \overline{\pi_H(r)_1}$ and $\pi_H(r)_1 \pi_H(r_3)$  are loops beginning at $Hg^{-1}$, therefore the paths $\pi_H(q) \pi_H(r)_1 \pi_H(r)_2 \overline{ \pi_H(r)_1} \pi_H(s)$ and $\pi_H(q) \pi_H(r)_1 \pi_H(r)_3 \pi_H (s)$ are loops beginning at $H \cdot 1$. It follows that the labels of these loops are elements of $H$. However, by construction
 $Lab( \pi_H(r)_1 \pi_H(r)_2 \overline{\pi_H(r_1)})  \equiv Lab(t_1 t_2 \bar t_1)=h_1$ and $Lab(\pi_H(r)_1 \pi_H(r)_3) \equiv Lab(t_1)Lab(t_3)=h_2$. As $Lab(\pi_H(q))=g^{-1}$ and $Lab(\pi_H(s))=g$, it follows that $g^{-1}h_1g \in H$ and  $g^{-1}h_2g \in H$. As $t_1t_3$ is a proper subpath of $t$, it follows that $|h_2| \le |t_1t_3| < |t| = |h|$. By construction
 $|t_1t_2| \le M^2 +1$, hence $|h_1|=|t_1t_2 \bar t_1| < 2 M^2 +1$, proving the lemma.
 
\end{proof}

We end this section by a brief discussion of the solvability of  membership problem for quasiconvex subgroups of negatively curved groups which was shown, for example, by the author in \cite{Gi2}, by Ilya Kapovich in \cite{Ka}, by Farb in \cite{Fa} and by Kharlampovich, Miasnikov and Weil in \cite{K-M-W}. This problem is undecidable in general case.

The author's solution in \cite{Gi2} uses the following definition.

\begin{definition} 
The geodesic core of $Cayley(G,H)$ is the union of all closed geodesics
in $Cayley(G,H)$ beginning at the vertex $H \cdot 1$.
We denote it $Core(G,H)$.
\end{definition}

 The geodesic core was recently reintroduced in \cite{K-M-W}, where it was called "a Stallings graph". 
\\
\\
By definition $Core(G,H)$ is an automaton accepting the language of geodesic representatives of $H$ in $Cayley(G)$ with the initial and final state being the base point $H \cdot 1$ (cf., for example, \cite{Ep}, p. 7). 

Remark \ref{R:GeodesicsInCosetGraph} implies the following observation, made in \cite{Gi2}.
 
\begin{remark}\label{R:GeodCoreFinite}
 A subgroup $H$  of a group $G$ is  $K$-quasiconvex 
if and only if  $Core(G,H)$ belongs to 
 $N_K(H \cdot 1) \subset Cayley(G,H)$. 
\end{remark}

It follows that if $G$ is finitely generated and $H$ is quasiconvex in $G$, then $Core(G,H)$ is finite, hence it is a finite state  automaton for the  language of geodesics in $H$. If, in addition, $G$ has a solvable word problem,  the membership problem is solvable for any $K$-quasiconvex subgroup of $H$. Indeed, given a word  $w$ in $X^*$ use the solution to the word problem in $G$ to find a geodesic representative $w_0$ of $w$ in $G$. Then check if the finite state automaton $Core(G,H)$ accepts $w_0$.

\section{An algorithm for computing the weak width of a quasiconvex subgroup of a negatively curved group and deciding its almost malnormality}

As $H$ acts on $Cayley(G,H)$ by right multiplication, the orbits of this action form a partition of the coset graph. These $H$-orbits are exactly the double cosets $HgH$, because $g_0 \in HgH$ if and only if $g_0=h_1gh_2$ for some $h_1$ and $h_2$ in $H$, which happens if an only if the vertices $Hg$ and $Hg_0$ in $Cayley(G,H)$ are connected by a path $p$ with $Lab(p) \in H$.
Remark \ref{R:ShortRep} implies that the cardinality of the intersection of $H$ with its conjugate $g_1^{-1}Hg_1$ is constant on the $H$-orbit of $g_1^{-1}$, so
in order to determine the weak width of a subgroup $H$ of $G$, we need to decide the finiteness of the intersections $H \cap g^{-1} H g$ for all $g$ in distinct $H$-orbits in $Cayley(G,H)$. Even in the case when $H$ is  quasiconvex and $G$ is negatively curved this problem  seems to be difficult, because it was shown by Arzhantseva in \cite{Ar} that an infinite index quasiconvex $H$ in a negatively curved $G$ has infinitely many distinct double cosets $HgH$.  However, Lemma 1.2 from \cite{G-M-R-S} states that if $H$ is $K$-qasiconvex and $G$ is $\delta$-negatively curved then the intersection  $H \cap g^{-1} H g$ might be infinite only if the $H$-orbit of $g$ contains a representative shorter than $2K + 2 \delta$. This lemma also appeared in \cite{Ar}. Here is the exact statement.

\begin{lemma}\label{L:SmallIntersection}
Let $H$ be a $K$-quasiconvex  subgroup of a
$\delta$-negatively curved group $G$ and let $g$ be an element in $G$. If every element of  
the double coset $HgH$ is longer than $2K + 2 \delta$, then the intersection
$H \cap g^{-1}Hg$ consists of elements shorter than $2K+8 \delta+2$, hence it is finite.
\end{lemma}

Lemma \ref{L:SmallIntersection} implies that if $H$ is $K$-quasiconvex and $G$ is $\delta$-negatively curved, then $WeakWidth(G,H)$ is finite. Indeed, as $G$ is finitely generated, there exists a finite number $N$ of elements in $G$ of length not greater  than $2K+ 2 \delta$, hence there exist at most $N$  elements $\{ g_i \in G \}$  such that the shortest representative of the double coset $Hg_iH$ is not longer than  $2K+ 2 \delta$. It follows from Lemma \ref{L:SmallIntersection}  that if a conjugate $g_0^{-1} H g_0$ has infinite intersection with $H$ then $g_0$ belongs to one of the double cosets $\{ Hg_iH | 1 \le i \le N \}$. However these double cosets need not be distinct, hence the number of conjugates  of 
$H$ by elements in different double cosets of $H$ which  have infinite intersection with $H$ is not greater than $N$. Therefore $WeakWidth(H,G) \leq N$.

These observations  can be refined further to give rise to a procedure for computing $WeakWidth(H)$.
\\
\\
\textbf{An algorithm for computing the weak width  of a quasiconvex subgroup of a negatively curved group.}

Input: a finite presentation $<X|R>$ for a group $G$ and a constant $\delta$ (not nessesary minimal) of negative curvature of $G$, a finite generating set for a subgroup $H \subset  G$ and a quasiconvexity constant $K$ (not necessarily minimal) of $H$.

Output: a finite list $L$ of the representatives of  distinct double cosets of $H$ such that for any $g \in L$ the intersection $H \cap g_i^{-1} H g$ is infinite.
$WeakWidth(G,H)$ is the length of the list $L$.

\begin{enumerate}
\item Make a list $L'$ of all words in $X^*$ which are geodesic representatives of distinct elements of $G$ with length not greater than $2K + 2\delta$.

As was mentioned already, this can be accomplished in finite time using the Dehn algorithm on the set of all reduced words in $X^*$ of length not greater than $2K + 2\delta$.

\item Form a sublist $L''$ consisting of elements of
 $L'$ which belong to distinct double cosets of $H$ in $G$.

To carry out this step note that the distance between any pair of elements $g_i$ and $g_j$ in $L'$ is not greater than $2 \cdot (2K + 2 \delta)$, hence a geodesic $\gamma$ in $Cayley(G,H)$ joining $Hg_i$ and $Hg_j$ is not longer than $4K + 4 \delta$. As $g_i$ and $g_j$ belong to the same double coset of $H$ if and only if $Lab(\gamma) \in H$ and, as was mentioned in section $3$, the membership problem for quasiconvex subgroups of negatively curved groups is solvable, this step can be completed in finite time.

\item Initialize the list $L$ by setting $L = 1_G$. For each $g_i \in L''$ if the cardinality of the intersection $H \cap g_i^{-1} H g_i$ is  infinite add $g_i$ to $L$. 

To carry out the third step use the algorithm developed  in section $3$.

\end{enumerate}

As was already mentioned, $H$ is almost malnormal in $G$ if and only if

$WeakWidth(G,H)=1$, hence the above algorithm also decides almost malnormality of $H$.

If $G$ is torsion-free, then malnormality is equivalent to almost malnormality, hence the above algorithm decides the malnormality of $H$. This observation implies the result of Baumslag, Miasnikov and Remeslennikovin that malnormality is decidable in free groups (\cite{B-M-R}). In general, malnormality is not decidable in negatively curved groups (\cite{B-W}). It is possible that it might be decidable for quasiconvex subgroups of negatively curved groups, but it  seems unlikely. Indeed, the above algorithm checks if the intersection  $H \cap g_i^{-1} H g$ is trivial for each $g_i^{-1} \in L''$. As the set of double cosets $HgH$ is infinite (\cite{Ar}), there are infinitely many representatives $s_i$ of distinct double cosets $Hs_iH$, which are not on the list $L''$.  Lemma \ref{L:SmallIntersection} states that for any such $s_i$ the intersection  $H \cap s_i^{-1} H s$ is finite and consists of elements shorter than $2K + 8 \delta +2$. As the membership problem for $H \cap s_i^{-1} H s$ is solvable, we can determine if this intersection  is trivial by checking if it contains finitely many elements shorter than $2K + 8 \delta +2$.
If $H$ is not malnormal, this procedure will detect
non-trivial intersection $H \cap s_i^{-1} H s$ after finitely many repetitions (the number of repititions does not have a good estimate), but if $H$ is malnormal, the procedure will not terminate in a finite time.
\\
\section{Algorithms for computing the width and the height of a quasiconvex subgroup of a negatively curved group.}

\begin{remark}\label{R:CompWidth}
In order to compute  $Width(G,H)$  we can use the algorithm from section $3$ repeatedly on the elements of the list $L$, constructed by the algorithm in section $4$, to build a list $L_1$ of representatives of  distinct double cosets of $H$ such that for any pair of elements $g_i$ and $g_j$  in $L_1$ the intersection $g_i^{-1} H g_i \cap g_j^{-1} H g_j$ is infinite.

Note that there might exist an element $g \notin L_1$     such that for some $g_j \in L_1$ the following holds: $g$ is a shortest representative of the coset $Hg$, $g \in Hg_jH, Hg \neq Hg_j$ but
for all $g_i \in L_1$ all the intersections  $g_i^{-1} H g_i \cap g^{-1} H g$ are infinite. In general, there might be even infinitely many elements like $g$, and all of them might be contributing to $Width(G, H)$, however, if $H$ is $K$-quasiconvex and $G$ is $\delta$-negatively curved, Remark \ref{R:FindingWidth} shows that there are only finitely many of them. In order to list all such elements, we use the proof of Lemma 1.3 in \cite{G-M-R-S} to deduce the following fact.
\end{remark}

\begin{lemma}\label{L:Width}
Let $g_i \in L_1$ be a non-trivial element of $G$.   Let $g$ be a shortest representative of the coset $Hg$ such that $Hg \neq Hg_i, g \in Hg_iH$ and $g^{-1} H g \cap g_i^{-1} H g_i$ is infinite.   Then there exists a decomposition $g_ig^{-1}=h_is_ik_i$ with $h_i \in H, k_i \in H, |s_i| \le 2K +3 \delta, |k_i| < 6K +21 \delta$ and $h_is_ik_i$ shortest possible.
\end{lemma}
\begin{proof}
As  $g_i \in L_1$, it follows that $|g_i| \le 2K + 2 \delta$ and $g_i$ is a shortest representative of the double coset $Hg_iH$, so  $|g| \ge |g_i|$.

As  $g^{-1} H g \cap g_i^{-1} H g_i$ is infinite, it follows that $g g_i^{-1} H g_i g^{-1} \cap H$ is infinite, so Lemma \ref{L:SmallIntersection} implies that $g_i g^{-1} \in HsH$ with $|s| \le
2K + 2 \delta$.

Setting $g=g_n$, the proof of Lemma 1.3 in \cite{G-M-R-S} can be applied to the products $g_i g^{-1}$. Then the existence of the desired decomposition $g_ig^{-1}=h_is_ik_i$ follows from the first $3$ lines of the proof of Lemma 1.3 and the Proposition on the fourth line.
\end{proof}

\begin{remark}\label{R:FindingWidth} 
 Lemma \ref{L:Width} suggests the following way for listing all the elements $g$ described in Remark \ref{R:CompWidth}. Let $\gamma_1 \gamma_2 \gamma_3$ be a triangle in $Cayley(G)$ with $\gamma_1$ beginning at the base point $1_G$, $Lab(\gamma_1) \equiv h_is_ik_i$, $Lab(\gamma_2) \equiv g$ and $Lab(\gamma_3) \equiv g_i^{-1}$. Note that $\gamma_2$ and $\gamma_3$ are geodesics, but $\gamma_1$ might not be one. Consider the projection of that triangle in $Cayley(G,H)$. Then $\pi_H(\gamma_1)$ begins at the base point $H \cdot 1$, and ends at $H \cdot Lab(\gamma_1) = H \cdot h_is_ik_i =H \cdot s_ik_i$. Lemma \ref{L:Width} states that $|s_ik_i| \le |s_i| + |k_i| \le (2K + 3 \delta) +(6K + 21 \delta) \le 8K + 24 \delta$. As the projection map does not increase distances, it follows that $H \cdot s_1h_2 \in N_{(|s_i| + |k_i|)}(H \cdot 1) \subset N_{(8K + 24 \delta)}(H \cdot 1)$. As $g_i$ is a shortest representative of the double coset $Hg_iH$ and $\pi_H(\overline{\gamma_3})$ begins at $H \cdot 1$, it follows that $|\pi_H(\overline{\gamma_3})|=|\overline{\gamma_3}|=|g_i| \le 2K + 2 \delta$.
Hence $g^{-1}$ is a label of a shortest geodesic in $Cayley(G,H)$ beginning at $H \cdot g_i$ and ending in 
$N_{(8K + 24 \delta)}(H \cdot 1)$. As the latter neighborhood is finite, there are finitely many potential $g$, so we can list them in a finite list $A_i$. Afterwards, for any $g \in A_i$ we check if
$g \in Hg_iH$, then we check if $g$ is a shortest representative of the coset $Hg$, then we check if $Hg \neq Hg_i$ and we conclude by checking if $g^{-1} H g \cap g_i^{-1} H g_i$ is infinite. If $g$ passes all these tests, we check for all $g_j \in L_1$ if the intersection  $g_j^{-1} H j_i \cap g^{-1} H g$ is infinite. If it is, $g$ might be contributing to $Width(G,H)$.
\end{remark}

This discussion can be summarized as follows.

\textbf{An algorithm for computing the width  of a quasiconvex subgroup of a negatively curved group.}

Input: a finite presentation $<X|R>$ for a group $G$ and a constant $\delta$ (not nessesary minimal) of negative curvature of $G$, a finite generating set for a subgroup $H \subset  G$ and a quasiconvexity constant $K$ (not necessarily minimal) of $H$.

Output: a finite list $L_w$ of the representatives of  distinct cosets of $H$ such that for any $g_i \in L_w$ and $g_j \in L_w$ the intersection $g_j^{-1}H g_j\cap  g_i^{-1} H g$ is infinite. $Width(G,H)$ is the length of the list $L_w$.
 
\begin{enumerate}

\item Run the algorithm for computing the weak width  of a quasiconvex subgroup of a negatively curved group, producing the list $L$. If $L = 1_G$, stop and output $L_w=1_G$.
\item If $|L| > 1$ modify $L$, using the algorithm from section $3$, as follows: for any $g_i \in L$ check if for any $g_j \in L$ with $j >i$ the intersection $g_i^{-1} H g_i \cap g_j^{-1} H g_j$ is finite. If positive, remove $g_j$ from the list $L$.
The resulting list $L_1$ consists of representatives of  distinct double cosets of $H$ such that for any pair of elements $g_i$ and $g_j$  in $L_1$ the intersection $g_i^{-1} H g_i \cap g_j^{-1} H g_j$ is infinite.

\item  Using Remark \ref{R:FindingWidth} for each $g_i \in L_1$ construct a finite set $A_i \subset G$  such that any $g_{i,j} \in A_i$ has the following properties: $g_{i,j} \in Hg_iH$, $g_{i,j}$ is a shortest representative of the coset $Hg_{i,j}$,  $Hg_{i,j} \neq Hg_i$ and $g_{i,j}^{-1} H g_{i,j} \cap g_i^{-1} H g_i$ is infinite. 

\item Initialize $L_w = L_1$. Let $A= \cup A_i$. If $A=\emptyset$ stop and output $L_w$. Otherwise enlarge the list $L_w$  as follows. For any $a_i \in A$,  using the algorithm from section $3$, check if for all $g_j \in L_1$ the intersections $a_i^{-1} H a_i \cap g_j^{-1} H g_j$ are infinite. If positive, add $a_i$ to the list $L_w$. Output $L_w$. It has the desired properties by construction.
\end{enumerate}

This construction shows that there is no obvious relationship between the width and the weak width even for a quasiconvex subgroup $H$ of a negatively curved group $G$. We can only observe that $|L_1| \le |L|=WeakWidth(G,H)$ and
$Width(G,H)=|L_w| \ge |L_1|$.

As was mentioned already, $H$ is almost malnormal in $G$ if and only if $Width(G,H)=1$, hence the above algorithm also decides the almost malnormality of $H$.
\\

In order to compute $Height(G,H)$ note that the list $L_w$ contains all elements $g_i$ in different cosets of $H$ such that the intersections $g_i^{-1} H g_i \cap  H$ is infinite, so we need to determine a sublist $L_h \subseteq L_w$ such that the mutual intersection of the conjugates of $H$ by the elements in $L_h$ is infinite.
We can do it as follows.

\textbf{An algorithm for computing the height  of a quasiconvex subgroup of a negatively curved group.}

Input: a finite presentation $<X|R>$ for a group $G$ and a constant $\delta$ (not nessesary minimal) of negative curvature of $G$, a finite generating set for a subgroup $H \subset  G$ and a quasiconvexity constant $K$ (not necessarily minimal) of $H$.

Output: a finite list $L_h$ of the representatives of  distinct cosets of $H$ such that the collection of conjugates of $H$ by all elements of $L_h$ has infinite intersection. $Height(G,H)$ is the length of the list $L_h$.
\begin{enumerate}

\item Run the algorithm for computing the width  of a quasiconvex subgroup of a negatively curved group, producing the list $L_w$. If $L_w = 1_G$, stop and output $L_h=1_G$.
\item If $|L_w| > 1$, hence $L_w= \{1_G, g_2, \cdots, g_n \}$, initialize $L_h =1_G$. By definition of $L_w$ the intersection $H \cap g_2^{-1}Hg_2$ is infinite. Add $g_2$ to $L_h$. If $n=2$ stop and output $L_h$.

\item If $n > 2$  enlarge $L_h$ by adding to it elements of $L_w$ in such way that the intersection of all conjugates of $H$ by elements of $L_h$ remains infinite. Proceed inductively as follows. For $g_3 \in L_w$, using the discussion in Remark \ref{R:Algorithm For Pairs}, check if the intersection $H \cap g_2^{-1}Hg_2 \cap g_3^{-1}Hg_3$ is infinite. If positive, add $g_3$ to $L_h$. Assume that for $i <j \le n$ we have decided if $g_i \in L_w$ should be added to $L_h$ and the intersection of all conjugates of $H$ by elements in $L_h$ is infinite. Check if the intersection of the aforementioned intersection with $g_j^{-1} H g_j$ is infinite. If positive, add $g_j$ to $L_w$. Repeat until $j=n$.
\item Stop and output $L_h$. 

\end{enumerate}

As $H$ is almost malnormal if and only if $Height(G,H)=1$ this algorithm also decides almost malnormality of $H$.

\begin{remark}\label{R:MoreIntersPairs}As was already mentioned, in order to determine if the intersection of a family of conjugates is infinite using the discussion in Remark \ref{R:Algorithm For Pairs}, we should provide a finite generating set for that intersection. It was shown by the author in \cite{Gi2} that a $K$-quasiconvex subgroup $H$ of a finitely generated group $G$ is generated by a set of elements  not longer than $2K +1$. Such generating set can be effectively listed if $H$ has a solvable generalized word problem in $G$, which is the case if $G$ is negatively curved. So it is sufficient to find a quasiconvexity constant (not necessarily the minimal one) of the aforementioned intersection.
\end{remark}

First, we find a quasiconvexity constant for each member of the aforementioned family of conjugates.

\begin{lemma}\label{L:QuasiconvexityConstantForConjug}
Let $H$ be a $K$-quasiconvex subgroup of a $\delta$-negatively curved group $G$. For any $g \in G$ the conjugate $g^{-1} H g$  is $K_g$-quasiconvex with $K_g = K + 2 \delta + 2 |g|$.
\end{lemma}
\begin{proof}
Let $g \in G$, let $h \in H$, and let $\gamma$ be a geodesic in $Cayley(G)$ beginning at $1_G$ with $Lab(\gamma)=g^{-1}hg$. We need to show that $\gamma \subset N_{K_g}(g^{-1}Hg) \subset Cayley(G)$. Consider a path $pqr \subset Cayley(G)$ beginning at $1_G$ such that $p, q $ and $r$ are geodesics, $Lab(p)= g^{-1}, Lab(q) = h$ and $Lab(r)=g$. Then $pqr$ and $\gamma$ have the same terminal vertex, so $pqr \overline{\gamma}$ is a geodesic $4$-gon, hence $\gamma \subset N_{2\delta}(pqr) \subset N_{2 \delta + |g|} (q)$. Let $v$ be a vertex in $\gamma$ and let $w$ be a vertex in $q$ such that $|v, w| \le 2 \delta + |g|$.
Let $q_1q_2$ be a decomposition of $q$ with $q_1$ terminating at $w$. As $H$ is $K$-quasiconvex, there exists a path $s$ shorter than $K$, beginning at $w$, such that $Lab(q_1s) \in H$. Let $t$ be a path beginning at the endpoint of $s$ with $Lab(t)=g$. Then the endpoint $u$ of $t$ is $Lab(pq_1st) \in g^{-1}Hg$ and $|v, u| \le |v, w| + |w, u| \le (|g| + 2 \delta) + (K  + |g|)=K_g$, proving the lemma.
\end{proof}

To find a quasiconvexity constant of the intersection of a finite family of quasiconvex subgroups it is sufficient to determine a quasiconvexity constant (not necessersly minimal) of the intersection of two quasiconvex subgroups. We will do that following the proof of Lemma 2.1 in \cite{Gi3}.

\begin{lemma}\label{L:QusiconvConstantForAnIntersOfaPair}
Let $A$ and $B$ be $K$-quasiconvex subgroups of a finitely generated group $G$. Let $M_A$ be the number of vertices in $N_K(A \cdot 1) \subset Cayley(G,A)$ and let $M_B$ be the number of vertices in $N_K(B \cdot 1) \subset Cayley(G,B)$. Then the intersection $A\cap B$ is $K_0$-quasiconvex in $G$ with $K_0 = M_A \cdot M_B$.
\end{lemma}
\begin{proof}
The proof of Lemma 2.1 demonstrates that $Core(G, A \cap B)\subset Cayley(G, A \cap B)$ embeds into $N_K(A \cdot 1) \times N_K(B \cdot 1) \subset Cayley(G,A) \times Cayley (G, B)$. Hence $Core(G, A \cap B)$ has no more than $M_A \cdot M_B$ vertices. Then, as $Core(G, A \cap B)$ is connected and contains $(A \cap B) \cdot 1$,  $Core(G, A \cap B)$ embeds in $N_{M_A \cdot M_B} (A \cap B) \cdot 1$. Therefore lemma follows from Remark \ref{R:GeodCoreFinite}.
\end{proof}
Note that Remark \ref{R:MoreIntersPairs} and Lemma \ref{L:QusiconvConstantForAnIntersOfaPair} can be used throughout the paper instead of Theorem \ref{T:GenSerFotINtersections}.

\section{Examples}

The following examples demonstrate that $WeakWidth(H,G)$, $Width(H,G)$, and $Height(H,G)$ are distinct invariants of the conjugacy class of $H$ in $G$.

\begin{example}\label{E:WWneqW}
Let $F$ be a free group of rank $4$ generated by the elements $x_1,x_2, x_3, x_{4}$, let $G = <F, t|t^4 =1, t^{-1} x_i t = x_{(i+1)mod4}| 1 \le i\le 4>$, and let $H_1=<x_1,x_2>$. We claim that $WeakWidth(H_1,G)=3$, but $Height(H_1,G)=Width(H_1, G)=2$.

In order to prove the claim we will list all  conjugates of $H_1$ in $G$ by elements in distinct cosets of $H_1$ and by elements in distinct double cosets of $H_1$ which have non-trivial intersection with $H_1$.

For $1\le i \le 4$ define $H_i = <x_i, x_{(i+1)mod4}| > = t^{(-i+1)}H_1 t^{(i-1)}$. As $t^i \notin F$ for $i \not\equiv 0\pmod{4}$, the subgroups $\{H_i| 1\le i \le 4 \}$ are conjugates of $H_1$ by elements in different double cosets of $H_1$. Note that $H_2 \cap H_1 = <x_2> ,H_4 \cap H_1 =<x_1>, H_3 \cap H_1 = <1>$, and $H_2 \cap H_4 = <1>$.  Hence $WeakWidth(H_1,G) \ge 3$,  $Height(H_1,G) \ge 2$, and $Width(H_1, G) \ge 2$.
        
In order to determine how other conjugates of $H_1$ intersect, consider $g \in G$  such that the intersection $g^{-1}H_1g \cap H_1$ is non-trivial. As we are interested only in conjugates of $H_1$ by elements in different cosets of $H_1$, we can assume that $g$ is a shortest element in the coset $H_1g$.

 As $t$ normalizes $F$, it follows that $g=wt^k, 0 \le k \le 3$, with $w$ a reduced word in $F$. If $w$ is trivial, then $g^{-1}H_1g= t^{-k}H_1t^k = H_{1+k}$, and the intersection pattern of the subgroups $\{H_i|1 \le i \le 4\}$ is described above.
 
 If $w$ is non-trivial, let $v \in H_1$ be a 
 non-trivial reduced word such that $g^{-1}vg = (t^{-k}w^{-1}) v (w t^k)  \in H_1$. Then $w^{-1} v w \in t^k H_1 t^{-k}  =t^{-(4-k)}H_1t^{4-k}  = H_{(1-k)mod4}$. As $w$ and $v$ are reduced words in a free group $F$, there exist decompositions $w \equiv w_1w_2$ and $v \equiv w_1v_0w_1^{-1}$ (where $\equiv$ denotes  equality of words) with 
 
 $w^{-1} v w = (w_2^{-1}w_1^{-1})(w_1v_0w_1^{-1})(w_1w_2) = w_2^{-1}v_0w_2$, where $w_2^{-1}v_0w_2$ is a reduced word in $H_{(1-k)mod4}$. Then $v_0 \in H_{(1-k)mod4}$ and  $w_2 \in H_{(1-k)mod4}$. As  $v \in H_1$, it follows that $w_1 \in H_1$ and $v_0 \in H_1$.
However, as $g=wt^k =w_1w_2t^k$ is shortest in the coset 
 $H_1g$, $w_1$ should be trivial. Hence $w=w_2 \in H_{(1-k)mod4}$. As a non-trivial word $v_0$ belongs to $H_1 \cap H_{(1-k)mod4}$, it follows that $(1-k)\pmod{4}$ is equal to either $1,2$ or $4$. Therefore, if $(1-k) \pmod4 \equiv3$, (so $k=2$), then for any $r \in F$ the intersection $ (rt^2)^{-1}H_1(rt^2) \cap H_1$ is trivial. 

If $(1-k)\pmod{4} \equiv1$ then $w=w_2 \in H_1$, contradicting again the fact that $g$ is shortest in the coset $H_1g$. Hence either $(1-k)\pmod{4} \equiv 2$ and $k=3$, or $(1-k)\pmod{4} \equiv4$ and $k=1$.

If $k=3$, then $g=wt^3$ with $w \in H_2$. Note that the  elements of the infinite collection  $\{(wt^3)^{-1}H_1(wt^3)| w \in H_2 \}$, which are conjugates by elements in different cosets of $H_1$, intersect each other trivially. Indeed, consider $w_0 \in H_2$ and $w \in H_2$ such that the intersection $(t^{-3}w^{-1})H_1(wt^3) \cap (t^{-3}{w_0}^{-1})H_1(w_0t^3)$  is non-trivial. Then the intersection $H_1 \cap (w_0t^3)(t^{-3}w^{-1})H_1(wt^3)(t^{-3} {w_0}^{-1})$ is  non-trivial. As $H_1$ is malnormal in $F$, it follows that  $w_0w^{-1}=(w_0t^3)(t^{-3}w^{-1}) \in H_1$, so the elements $wt^3$ and $w_0t^3$ belong to the same coset of $H_1$. Therefore the family of the conjugates
$\{(wt^3)^{-1}H_1(wt^3)| w \in H_2 \}$
does not contribute to $Width(H_1,G)$.

Similarly, if $k=1$, hence $g=ut$ with $u \in H_4$, the   elements of the infinite collection $\{(ut)^{-1}H_1(ut)| u \in H_4 \}$, which are conjugates by elements in different cosets of $H_1$,  intersect each other trivially.

Also for $w \in H_2$ and $u \in H_4$ the intersection $(t^{-3}w^{-1})H_1(wt^3) \cap (t^{-1}u^{-1})H_1(ut)$ is trivial. Indeed, the cardinality of that intersection is equal to the cardinality of the intersection
$(ut)(t^{-3}w^{-1})H_1(wt^3)(t^{-1}u^{-1}) \cap H_1$. However, $(wt^3)(t^{-1}u^{-1})=wt^2u^{-1}=(w(t^2u^{-1}t^{-2})t^2 = rt^2$ with $r \in F$, and we have mentioned above that for all $r \in F$ the intersection $ (rt^2)^{-1}H_1(rt^2) \cap H_1$ is trivial. So the infinite family of conjugates $\{(ut)^{-1}H_1(ut)| u \in H_4 \}$
does not contribute to $Width(H_1,G)$, therefore $Height(H_1,G)= Width(H_1,G)=2$. 

Note that for any $w \in H_2$,  $wt^3=t^3(t^{-3}wt^3) \in t^3H_1 \subseteq H_1t^3 H_1$, so the conjugates of $H_1$ by those elements do not contribute to the weak width of $H_1$. Similarly, all the elements $\{(ut)^{-1}| u \in H_4 \}$ belong to the double coset  $H_1tH_1$, so the conjugates of $H_1$ by those elements  do not contribute to the weak width of $H_1$ either.
Therefore, $WeakWidth(H,G)=3$.
 
$\hfill\square$
\end{example}

\begin{example}\label{WneH}
Let $G$ be as in Example \ref{E:WWneqW}, and let $L_1=<x_1,x_2, x_3>$. We claim that $WeakWidth(L_1,G)=Width(L_1, G)=4$, but $Height(L_1,G)=3$.

For $1\le i \le 4 $ define $L_i = <x_i, x_{(i+1)mod4}, x_{(i+2)mod4} > = t^{(-i+1)}L_1 t^{(i-1)}$.  
As $t^i \notin F$ for $i \not\equiv 0\pmod{4}$, the subgroups $\{L_i| 1\le i \le 4 \}$ are conjugates of $L_1$ by elements in different double cosets of $L_1$. 
By observation, the elements of the set $\{ L_i| 1 \le i \le 4 \}$ have infinite pairwise intersections, hence
$WeakWidth(L_1,G) \ge 4$ and $Width(L_1, G) \ge 4$. Also the intersection $\bigcap\limits_{i =1}^3 L_i$ is infinite, so $Height(L_1,G) \ge 3$. Note also that
the intersection $\bigcap\limits_{i =1}^4 L_i$ is trivial. 

Using the same argument as in Example \ref{E:WWneqW} we can show that there are only three families of subgroups, which are conjugates of $L_1$ by  elements in different cosets of $L_1$ in $G$, intersecting $L_1$ non-trivially. These elements are $\{(wt^3)| w \in L_2 \}, \{ ut| u \in L_4 \}$, and $ \{ st^2| s \in L_3 \}$. Just as in Example \ref{E:WWneqW}, the malnormality of $L_1$ in $F$ implies that the conjugates in each family intersect each other trivially, hence $Width(L_1, G)=4$. Also, as in Example \ref{E:WWneqW}, these elements belong to the same double cosets of $L_1$ as $t^3, t$, and $t^2$, respectively, so $WeakWidth(L_1,G)=4$.

Suppose $Height(H,G) \ge 4$. Then there are $3$  conjugates $M_2, M_3$, and $M_4$ of $L_1$ by elements in different cosets of $L_1$ such that the intersection $L_1 \cap (\bigcap\limits_{i=2}^4 M_i)$ is infinite. The preceding paragraph implies that the $M_i$'s must come one from each of  
the families of conjugates of $L_1$ described above, i.e. $M_2, M_3$ and $M_4$ are conjugates of $L_1$ by $wt^3, ut$, and $st^2$ respectively, with $w \in L_2, u \in L_4$, and $s \in L_3$. Let $h_1, h_2, h_3$, and $h_4$ in $L_1$ be such that $h_4 =t^{-3}w^{-1}h_1wt^3=t^{-2}s^{-1}h_2st^2=t^{-1}u^{-1}h_3ut$. Note that $t^{-3}wt^3 \in L_1, t^{-3}h_1t^3 \in L_4, t^{-2}st^2 \in L_1, t^{-2}h_2t^2 \in L_3, t^{-1}ut \in L_1$, and $t^{-1}h_3t \in L_2$. Then $t^{-3}w^{-1}h_1wt^3=r^{-1}_1q_1r_1$ with $r_1 \in L_1$ and $q_1 \in L_4$, $t^{-2}s^{-1}h_2st^2=r_2^{-1}q_2r_2$ with $r_2 \in L_1$ and $q_2 \in L_3$, and $t^{-1}u^{-1}h_3ut = r_3^{-1}q_3r_3$ with $r_3 \in L_1$ and $q_3 \in L_2$. 
As  $r^{-1}_1q_1r_1 = r_2^{-1}q_2r_2 = r_3^{-1}q_3r_3$,
 it follows that $q_2 = l_1^{-1}q_1l_1=l_2^{-1}q_3l_2$ with $l_1$ and $l_2 $ in $L_1$. We can assume that all the words $l_1, l_2, q_1, q_2$, and $q_3$ are reduced. Then, as in Example \ref{E:WWneqW}, there exist decompositions
$l_1 \equiv p_1p_2$ and $q_1 \equiv p_1q'_1p_1^{-1}$ such that $q_2 = (p_1p_2)^{-1}(p_1q'_1p_1^{-1})(p_1p_2)=p_2^{-1}q'_1p_2$, and $p_2^{-1}q'_1p_2$ is a reduced word in $F$.

 As 
$r^{-1}_1q_1r_1 = r_2^{-1}q_2r_2 = r_3^{-1}q_3r_3 = h_4 \in L_1$, it follows that $q_1 \in L_1 \cap L_4 =<x_1, x_2>, q_2 \in L_1 \cap L_3 = <x_1, x_3>$, and $q_3 \in L_1 \cap L_2 = <x_2, x_3>$. As $q_1 \in <x_1, x_2>$ and $q_2 \in <x_1, x_3>$, it follows that $q'_1 = x_1^n$ for $ n \in \textbf{N}$. 

Simirlarly, there exist decompositions $l_2 \equiv c_1c_2$ and $q_3 \equiv c_1q'_3c_1^{-1}$ such that $q_2 = (c_1c_2)^{-1}(c_1q'_3c_1^{-1})(c_1c_2)=c_2^{-1}q'_3c_2$, and $c_2^{-1}q'_3c_2$ is a reduced word in $F$. As $q_3 \in <x_2, x_3>$ and $q_2 \in <x_1, x_3>$, it follows that $q'_3 = x_3^m$ for $m \in \textbf{N}$. Then a conjugate of $q'_1 = x_1^n$ is equal to a conjugate of $q'_3 = x_3^m$ in a free group $F$. This can happen only if  $q'_1$ and $q'_3$ are trivial, hence $q_2$ is trivial. Therefore, the intersection of $L_1$ with all three families of conjugates is trivial, so $Height(L_1, G)=3$.

$\hfill\square$
\end{example}

\end{document}